\documentclass[10pt]{article}
\RequirePackage[T1]{fontenc}	
\RequirePackage{textcomp}	
\RequirePackage[utf8]{inputenc}	
\RequirePackage{lmodern}				
\RequirePackage[french,english]{babel}		
\RequirePackage[babel]{csquotes}	

\usepackage{color}
\usepackage{amsfonts}
\usepackage{amsmath}
\usepackage{amssymb}
\usepackage{graphicx}
\usepackage{mathptmx}
\usepackage{subfigure}
\usepackage{float}
\usepackage{parskip}
\RequirePackage[colorlinks, hyperindex, plainpages=false]{hyperref}
\usepackage{amsmath}

\def\P{{\mathbb P}}
\def\E{{\mathbb E}}

\def\R{{\mathbb R}}

\def\1{{\mathbf 1}}

\newtheorem{theorem}{Theorem}[section]
\newtheorem{lemma}[theorem]{Lemma}
\newtheorem{corollary}[theorem]{Corollary}
\newtheorem{proposition}[theorem]{Proposition}

\newtheorem{theo}[theorem]{Theorem}

\newenvironment{proof}[1][Proof.]{\textbf{#1} }{\hfill $\blacksquare$}
\def\beq{\begin{equation}}
\def\eeq{\end{equation}}
\newcommand{\bei}{\begin{itemize}}
\newcommand{\eei}{\end{itemize}}
\newcommand{\ben}{\begin{enumerate}}
\newcommand{\een}{\end{enumerate}}
\newcommand{\beqn}{\begin{eqnarray}}
\newcommand{\beqnn}{\begin{eqnarray*}}
\newcommand{\eeqn}{\end{eqnarray}}
\newcommand{\eeqnn}{\end{eqnarray*}}
\newcommand{\brm}{\begin{rmk}}
\newcommand{\erm}{\end{rmk}}

\begin{document}

\title{A new approximation of the Height process of a CSBP}
\author{
Ibrahima~Dram\'{e}  \footnote{  \scriptsize{Université Cheikh Anta Diop de Dakar, FST, LMA, 16180 Dakar-Fann, S\'en\'egal. iboudrame87@gmail.com
}}
}

\maketitle

\begin{abstract}
We code Galton-Walton trees by a continuous height process, in order to give a precise meaning to the convergence of forests of trees. This allows us to establish the convergence of the forest of genealogical trees of the branching process of a large population towards the genealogical trees of the limiting continuous state branching process (CSBP). The approximation considered here is new, compared to that which has been studied in \cite{DP2}.
\end{abstract}

\vskip 1mm

\section{Introduction}
Continuous state branching processes (or CSBP in short) are the analogues of Galton-Watson (G-W) processes in continuous time and continuous state space. Such classes of processes have been introduced by Jirina \cite{Ji} and studied by many authors included Grey \cite{Grey}, Lamperti \cite{Lam1},  to name but a few. These processes are the only possible weak limits that can be obtained from sequences of rescaled G-W processes, see  Lamperti \cite{Lam2}. 

While rescaled discrete-time G-W processes converge to a CSBP,  it has been shown in Duquesne and Le Gall \cite{DLG} that the genealogical structure of the G-W processes converges too. More precisely, the corresponding rescaled sequences of discrete height process, converges to the height process in continuous time that has been introduced by Le Gall and Le Jan in \cite{lGlJ}. For the approximation by continuous time 
generalized G-W processes we refer to our recent paper \cite{DP}.

Some work has been also devoted recently to the description of the genealogy of  generalized CSBPs, see Dramé and Pardoux  \cite{DP2} and Dram\'{e} et al. in \cite{DPS} for the case of continuous such processes and Li, Pardoux and Wakolbinger \cite{LiPW} for the general case. In \cite{DP2} Dramé and Pardoux give an approximation of the Height process of a continuous state branching process in terms of a stochastic integral equation with jumps, which is well suited for the case of generalized CSBPs. The present paper studies another approximation of the genealogy of a continuous time GW process to that of a generalized possibly discontinuous CSBP, under the same assumptions as \cite{DP2}. Note that, it would be interesting to prove a priori that the two approximations must have the same limit.

The organization of the  paper is as follows : In Section 2 we recall some basic definitions and notions concerning branching processes. 
Section 3 is devoted to the description of the discrete approximation of both the population process and the height process of its genealogical forest of trees. We prove the convergence of the height process.
 We shall assume that all random variables in the paper are defined on the same probability space $( \Omega,\mathcal{F},\mathbb{P})$. We shall use the following notations $\mathbb{Z}_+=\{0,1, 2,... \}$, $\mathbb{N}=\{1, 2,... \}$, $\mathbb{R}=(-\infty, \infty)$ and $\mathbb{R}_+=[0, \infty)$. For $x$ $\in$ $\mathbb{R}_+$, $[x]$ denotes the integer part of $x$.

\section{The Height process of a continuous state branching process}
\subsection{Continuous state branching process}
A CSBP is a $\mathbb{R}_+$-valued strong Markov process with the property that $\mathbb{P}_x$ denoting the law of the process when starts from $x$ at time $t=0$, $\mathbb{P}_{x+y}= \mathbb{P}_{x} \ast \mathbb{P}_{y}$. More precisely, a CSBP $X^x=(X_t^x, \ t\geqslant0)$ (with initial condition $X_0^x=x$) is a Markov process taking values in $[0, \infty]$, where $0$ and $\infty$ are two absorbing states, and satisfying the branching property; that is to say, it's Laplace transform satisfies 
\begin{equation*}
   \mathbb{E} \left[ \exp (-\lambda X_t^x)\right] = \exp \left\{ -x u_t(\lambda)\right\}, \quad \mbox{for} \ \lambda\geqslant0, 
\end{equation*}
for some non negative function $u_t(\lambda)$. According to Silverstein \cite{Sil}, the function $u_t$ is the unique nonnegative solution of the integral equation : $u_t(\lambda) = \lambda - \int_{0}^{t} \psi ( u_r(\lambda)) dr, $
where $\psi$ is called the branching mechanism associated with $X^x$ and is defined by
$ \psi(\lambda) = b\lambda +{c\lambda^2}+ \int_{0}^{\infty} (e^{-\lambda z}-1+\lambda z\mathbf{1}_{\{z\le 1\}}) \mu(dz)$, with $b \in \mathbb{R}$, $c\geqslant0$ and $\mu$ is a $\sigma$-finite measure which satisfies $\int_{0}^{\infty}(1\wedge z^2) \mu(dz) < \infty$. We shall in fact assume in this paper that 
\begin{equation*}
 {(\bf H)}  : \quad \int_{0}^{\infty} (z\wedge z^2) \mu(dz) < \infty \quad \mbox{and} \quad c>0 .
\end{equation*}
The first assumption implies in particular that the process $X^x$ does not explode and it allows is to write the last integral in the above equation in the following form 
\begin{equation}\label{PSIEXPRES}
 \psi(\lambda) = b\lambda +{c\lambda^2}+ \int_{0}^{\infty} (e^{-\lambda z}-1+\lambda z) \mu(dz).
\end{equation}
From Fu and Li \cite{FL} (see also the results in Dawson-Li \cite{DaLi}), we have 
\begin{equation}\label{XT}
X_t^x= x - b \int_{0}^{t} X_s^x ds + \sqrt{2c} \int_{0}^{t}  \int_{0}^{X_s^x} W(ds,du) + \int_{0}^{t}\int_{0}^{\infty}\int_{0}^{X_{s^-}^x}z \overline{M}(ds, dz, du),
\end{equation}
where $W(ds,du)$ is a space-time white nose on $(0,\infty)^{2}$, $M(ds, dz, du)$ is a Poisson random measure on $(0,\infty)^{3}$, with intensity $ds\mu(dz)du$, and $\overline{M}$ is the compensated measure of $M$. 
\subsection{The height process}
We shall also interpret below the function $\psi$ defined by \eqref{PSIEXPRES} as the Laplace exponent of a spectrally positive L\'{e}vy process $Y$. Lamperti \cite{Lam1} observed that CSBPs are connected to L\'{e}vy processes with no negative jumps by a simple time-change. More precisely, define
\begin{equation*}
A_s^x=\int_{0}^{s}X_t^x dt, \quad \tau_s=\inf\{t>0, \ A_t^x >s\} \quad \mbox{and} \quad Y(s)=X_{\tau_s}^x. 
\end{equation*}
Then $Y(s)$ is a Lévy process of the form until the first that it hits $0$
\begin{equation}\label{YLevy}
Y(s)= - b s + \sqrt{2c} B(s) + \int_{0}^{s}\int_{0}^{\infty}z \overline{\Pi}(dr, dz),
\end{equation}
where $B$ is a standard Brownian motion and $\overline{\Pi}(ds, dz)= \Pi(ds, dz)- ds\mu(dz)$, $\Pi$ being a Poisson random measure on $\mathbb{R}_{+}^{2}$ independent of $B$ with mean measure $ds\mu(dz)$. We refer the reader to \cite{Lam1} for a proof of that result. To code the genealogy of the CSBP, Le Gall and Le Jan \cite{lGlJ} introduced the so-called height process, which is a functional of a L\'{e}vy process with Laplace exponent $\psi$; see also Duquesne and Le Gall \cite{DLG}. In this paper, we will use the new definition of the height process $H$ given by Li et all in \cite{LiPW}. Indeed, if the Lévy process $Y$ has the form \eqref{YLevy}, then the associated height process is given by 
\begin{equation}\label{CHS1}
c H(s)= Y(s) - \inf_{0\leqslant r \leqslant s} Y(r) - \int_{0}^{s}\int_{0}^{\infty}\left(z + \inf_{r\leqslant u \leqslant s}Y(u)-Y(r) \right)^+\Pi(dr, dz),
\end{equation}
and it has a continuous modification. Note that the height process $H$ is the one defined in formula (1.4) in \cite{DLG}, i.e
$cH(s)=| \{\overline{Y}^s(r); \ 0\le r\le s \}|,$ where $\overline{Y}^s(r):= \inf_{r\le u\le s} Y(u)$  and $|A|$ denotes the Lebesgue measure of the set $A$.

\section{Approximation of the Height process}
In the following,  we consider a specific forest of Bellman-Harris trees, obtained by Poissonian sampling of the height process $H$. In other words, let $\alpha >0$ and we consider a standard Poisson process with intensity $\alpha$. We denote by $\tau_1^\alpha \le \tau_2^\alpha \le \cdots $ the jump times of this Poisson process. If $H$ is seen as the contour process of a continuous tree, consider the forest of the smaller trees carried by the vector $H(\tau_1^\alpha), H(\tau_2^\alpha),\cdots $. We have 

\begin{proposition} $(Theorem \ 3.2.1 \ in \ \cite{DLG})$
The trees in this forest are trees, which are distributed as the family tree of a continuous-time Galton-Watson process starting with one individual at time $0$ and such that : 

$\ast$ Lifetimes of individuals have exponential distributions with parameter  $\psi^{\prime} ( \psi^{-1}(\alpha))$;

$\ast$ The offspring distribution is the law of the variable $\eta$ with generating function : \quad
\[ \E(s^\eta)= s+ \frac{\psi((1-s)\psi^{-1}(\alpha))}{\psi^{-1}(\alpha) \psi^{\prime} ( \psi^{-1}(\alpha))}.\]
\end{proposition}

Let $N\geqslant1$ be an integer which will eventually go to infinity. In the next two sections, we choose a sequence $\delta_N \downarrow 0$ such that, as $N\rightarrow \infty$, 
\begin{equation*}
{(\bf A)}  : \quad \frac{1}{N} \int_{\delta_N}^{+\infty}  \mu(dz) \rightarrow 0.
\end{equation*}
This implies in particular that $\frac{1}{N} \int_{\delta_N}^{+\infty} z \mu(dz) \rightarrow 0.$
Moreover, we will need to consider
\begin{equation}\label{deltan}
 \psi_{\delta_N}(\lambda) =  {c\lambda^2}+ \int_{\delta_N}^{\infty} (e^{-\lambda z}-1+\lambda z) \mu(dz).
\end{equation}
We will also set $\alpha= \psi_{\delta_N}(N)$ in the limit of large populations. 

\subsection{A discrete mass approximation}\label{sec:GWcont}

The aim of this subsection is to set up a "discrete mass - continuous time" approximation of \eqref{XT} . To this end, we set 
\begin{equation*}
h_{N}(s)= s + \frac{\psi_{\delta_N}((1-s)N)}{N \psi_{\delta_N}^{\prime} (N)}, \quad \quad |s| \leqslant 1.
\end{equation*}
It is easy to see that $s\rightarrow h_{N}(s)$ is an analytic function in $(-1,1)$ satisfying $h_{N}(1)=1$ and $\frac{d^n}{ds^n}h_{N}(0)\ge 0, \  n \ge 0.$ Therefore $h_{N}$ is a probability generating function. and we have $h_{N}(s)= \sum_{\ell\ge0} \nu_{N}(\ell)s^\ell,  \ |s| \leqslant 1,$ where $\nu_{N}$ is probability measure on $\mathbb{Z}_+$. Fix $x>0$ the approximation of \eqref{XT} will be given by the total mass $X^{N,x}$ of a population of individuals, each of which has mass $1/N$. The initial mass is $X_0^{N,x}= [Nx]/N$, and $X^{N,x}$ follows a Markovian jump dynamics :  from its current state $k/N$,  
 \begin{equation*}
X^{N,x} \  \mbox{jumps to} \ 
\left\{
    \begin{array}{ll}
   \frac{k+\ell-1}{N} \ \mbox{at rate} \  \psi_{\delta_N}^{\prime} (N) \nu_{N}(\ell) k,  \ \mbox{for all} \ \ell\ge2; &\\\\  \frac{k-1}{N}  \quad  \mbox{at rate} \  \psi_{\delta_N}^{\prime} (N) \nu_{N}(0) k.
    
&
           \end{array}
           \right.
\end{equation*}

In this process, each individual dies without descendant at rate
\[ \frac{\psi_{\delta_N}(N)}{N}=cN+\int_{\delta_N}^\infty z\mu(dz)-\frac{1}{N}\int_{\delta_N}^\infty (1-e^{-Nz})\mu(dz),\]
dies and leaves two descendants at rate
$ cN+\frac{1}{N}\int_{\delta_N}^\infty \frac{(Nz)^2}{2}e^{-Nz}\mu(dz)$,
and finally dies and leaves $k$ descendants ($k\ge3$) at rate
$\frac{1}{N}\int_{\delta_N}^\infty\frac{(Nz)^k}{k!}e^{-Nz}\mu(dz).$ Let $\mathcal{D}([0,\infty), \mathbb{R}_+)$ denote the space of functions from $[0,\infty)$ into $\mathbb{R}_+$ which are right continuous and have left limits at any $t>0$. We shall always equip the space $\mathcal{D}([0,\infty),\mathbb{R}_+)$ with the Skorohod topology. The main limit proposition of this subsection is a consequence of Theorem 4.1 in \cite{DP}. 
\begin{proposition}\label{TH1}
 Suppose that Assumptions $(\bf H)$ is satisfied. Then, as $N\rightarrow +\infty$, $\{X_t^{N,x}, \ t\geqslant0 \}$ converges to $\{X_t^x, \ t\geqslant0 \}$ in distribution on $\mathcal{D}([0,\infty),\mathbb{R}_+)$, where $X^x$ is the unique solution of the SDE \eqref{XT}.  
\end{proposition}

 \subsection{The approximate height process}
In this section, we show that the rescaled exploration process of the corresponding Galton-Watson genealogical forest of trees, converges in a functional sense, to the continuous height process associated with the CSBP. We will first need to write precisely the evolution of $\{H^N(s), s \ge 0\}$, the height process of the forest of trees representing the population described in section 3. To this end, to any $\delta > 0$, we define
\[Y_{\delta}(s)= -\left(b+\int_{\delta}^\infty z\mu(dz) \right)s + \sqrt{2c} B(s) + \int_0^s \int_{\delta}^\infty z \Pi(dr,dz).\]
and we associate $H_\delta$ the exploration process defined with the Lévy process $Y_\delta$. In other words, we have suppressed the small jumps, smaller than $\delta$, i.e \eqref{CHS1} takes the following form
\begin{equation}\label{cHs}
cH_\delta (s)= Y_{\delta}(s)- \inf_{0\le r\le s} Y_{\delta}(r)- \int_0^s \int_{\delta}^\infty \left(z + \inf_{r\le u\le s} Y_{\delta}(u) - Y_{\delta}(r)\right)^+\Pi(dr,dz).
\end{equation}
We consider for each $N\ge1$ a Poisson process $\{P_s^N, s\ge 0\}$ with intensity $\psi_{\delta_N} (N)$ independent from $\{Y(s), s \ge 0\}$. We denote by $\tau_1^N \le \tau_2^N \le \cdots$ the jump times of this Poisson process. The height process $\{H^N(s),s \ge 0\}$ is simply the piecewise affine function of slope $\pm 2N$ passing through the values
\[0, H_{\delta_N}(\tau_1^N), \min_{s\in [\tau_1^N,\tau_2^N]} H_{\delta_N}(s), H_{\delta_N}(\tau_2^N), \min_{s\in [\tau_2^N,\tau_3^N]} H_{\delta_N}(s), \cdots, H_{\delta_N}(\tau_n^N), \min_{s\in [\tau_n^N,\tau_{n+1}^N]} H_{\delta_N}(s), \cdots  \]
see Duquesne and Le Gall \cite{DLG}. We are ready to state the main result of this paper. Recall the process $H$ defined in \eqref{CHS1}.
\begin{theo}\label{THP}
 For any $s > 0$, $H^N(s) \longrightarrow H(s)$ in probability, locally uniformly in $s$, as $N \rightarrow \infty$.
\end{theo}
To prove this theorem, we will proceed in several steps. So, for any $s > 0$, we define
\[Y^{ref}(s)=Y(s) - \inf_{0\le r\le s} Y(r) \quad \mbox{and} \quad Y_{\delta_N}^{ref}(s)=Y_{\delta_N}(s) - \inf_{0\le r\le s} Y_{\delta_N}(r) .\]
From now on, we do as if $Y^{ref}$  and $Y_{\delta_N}^{ref}$ were deterministic, only $P^N$ (and the $\tau_k^N$’s) are random. 

 A first preparation for the proof of Theorem \ref{THP} is
\begin{lemma}\label{GG}
For any $h\in\mathcal{C}(\R_+; [0,1])$,
\[\frac{1}{\psi_{\delta_N}(N)} \sum_{k=1}^{[\psi_{\delta_N}(N)s]} h(Y_{\delta_N}^{ref}(\tau_k^N)) \longrightarrow \int_0^s h(Y^{ref}(r))dr \quad in \ probability,\ as \ N\rightarrow \infty. \]
\end{lemma}
\begin{proof}
We have
\begin{align*}
\frac{1}{\psi_{\delta_N}(N)} \sum_{k=1}^{[\psi_{\delta_N}(N)s]} h(Y_{\delta_N}^{ref}(\tau_k^N))-  \int_0^s h(Y^{ref}(r))dr&= \frac{1}{\psi_{\delta_N}(N)} \sum_{k=1}^{[\psi_{\delta_N}(N)s]} h(Y_{\delta_N}^{ref}(\tau_k^N))- \frac{1}{\psi_{\delta_N}(N)}   \int_{[0,s]} h(Y_{\delta_N}^{ref}(r))dP_r^N\\& +\frac{1}{\psi_{\delta_N}(N)}   \int_{[0,s]} h(Y_{\delta_N}^{ref}(r))dP_r^N   -  \int_0^s h(Y_{\delta_N}^{ref}(r))dr\\&+
 \int_0^s h(Y_{\delta_N}^{ref}(r))dr- \int_0^s h(Y^{ref}(r))dr\\
 &=A_N(s)+B_N(s)+C_N(s).
 \end{align*}
 First $C_N(s)\rightarrow 0$ follows readily from $\sup_{0\le r\le s} \left|  h(Y^{ref}(r))- h(Y_{\delta_N}^{ref}(r)) \right| \rightarrow 0$, as $N\rightarrow \infty$, since $h$ is continuous and  $\sup_{0\le r\le s} \left|  Y^{ref}(r)- Y_{\delta_N}^{ref}(r) \right| \rightarrow 0$, as $N\rightarrow \infty$. Next we have 
$B^N(s)= \frac{1}{\psi_{\delta_N}(N)}   \int_{[0,s]} h(Y_{\delta_N}^{ref}(r)) [dP_r^N   -  \psi_{\delta_N}(N)dr ]$.
We have $\E[B^N(s)] = 0$, while $\mathbb{V}ar(B^N(s))=  \frac{1}{\psi_{\delta_N}(N)} \E  \int_0^s h(Y_{\delta_N}^{ref}(r))^2dr,$ which clearly tends to $0$ as $N\rightarrow \infty$, since $h$ is bounded and $\psi_{\delta_N} (N) \rightarrow \infty$. Consequently $B^N(s)\rightarrow 0$ in probability, as $N \rightarrow \infty$. It remains to consider $A^N$. Since  $0\le h(y) \le 1$,  $|A^N(s)|\le  \frac{1}{\psi_{\delta_N}(N)} |P_s^N   -  \psi_{\delta_N}(N)s | \longrightarrow 0$
$a.s.$ from the strong law of large numbers. The result follows.
 \end{proof}
 
 For any $N\ge 1$, $s>0$, we define
 \begin{equation}\label{314}
 K_N(s)= \frac{1}{2N} H_{\delta_N} (\tau_1^N)+  \frac{1}{2N} \sum_{k=1}^{[\psi_{\delta_N}(N)s]} \left\{ (H_{\delta_N}(\tau_k^N)- \min_{r\in [\tau_k^N,\tau_{k+1}^N] } H_{\delta_N}(r) )+  (H_{\delta_N}(\tau_{k+1}^N)- \min_{r\in [\tau_k^N,\tau_{k+1}^N] } H_{\delta_N}(r) ) \right\}
  \end{equation}
It is not hard to see that $K^N(s)$ is the time taken by the process $H^N$ to reach the point $H_{\delta_N} \left( \tau^N_{[\psi_{\delta_N}(N)s]}\right)$.

So we get by our construction that
\begin{equation}\label{ega}
H^N\left(K_N(s) \right)=H_{\delta_N} \left( \tau^N_{[\psi_{\delta_N}(N)s]}\right). 
\end{equation}
For the proof of Theorem \ref{THP} we will need the two following Propositions.
\begin{proposition}\label{tauNs}
For any $s>0$, $\tau^N_{[\psi_{\delta_N}(N)s]} \longrightarrow s$ $a.s$, as $N\rightarrow \infty$.
\end{proposition}
\begin{proof}
It is easy to see that $\tau^N_{[\psi_{\delta_N}(N)s]}= \frac{1}{\psi_{\delta_N}(N)} \sum_{k=1}^{[\psi_{\delta_N}(N)s]} \xi_k,$ where $(\xi_k)_{k\ge1}$ is an sequence of independent and identically distributed (i.i.d)$\sim Exp(1)$. The desired result follows easily from the law of large numbers.
\end{proof}
\begin{proposition}\label{KNs}
 For any $s > 0$, $K_N(s) \longrightarrow s$ in probability, as $N \rightarrow \infty$.
\end{proposition}
\begin{proof}
 Let us rewrite \eqref{314} in the form
  \begin{align*}
 K_N(s)&= \frac{1}{N} \sum_{k=1}^{[\psi_{\delta_N}(N)s]} (H_{\delta_N}(\tau_k^N)- \min_{r\in [\tau_k^N,\tau_{k+1}^N] } H_{\delta_N}(r) ) +  \frac{1}{2N} \left( H_{\delta_N} (\tau_1^N) +   \sum_{k=1}^{[\psi_{\delta_N}(N)s]} (H_{\delta_N}(\tau_{k+1}^N) -H_{\delta_N}(\tau_{k}^N)\right)\\&= K_1^N(s)+ K_2^N(s)+ K_3^N(s), \quad \mbox{with}
 \end{align*}
\[K_1^N(s)= \frac{1}{cN} \sum_{k=1}^{[\psi_{\delta_N}(N)s]} \left\{ (cH_{\delta_N}(\tau_k^N)- \min_{r\in [\tau_k^N,\tau_{k+1}^N] } cH_{\delta_N}(r) ) -  (Y_{\delta_N}^{ref}(\tau_k^N)- \min_{r\in [\tau_k^N,\tau_{k+1}^N] } Y_{\delta_N}^{ref}(r) ) \right\}, \]
\[ K_2^N(s)=  \frac{1}{cN} \sum_{k=1}^{[\psi_{\delta_N}(N)s]}   (Y_{\delta_N}^{ref}(\tau_k^N)- \min_{r\in [\tau_k^N,\tau_{k+1}^N] } Y_{\delta_N}^{ref}(r) ) \quad \mbox{and} \quad K_3^N(s)= \frac{1}{2N} H_{\delta_N} \left( \tau^N_{[\psi_{\delta_N}(N)s]}\right).  \]
A standard argument combined with Proposition \ref{tauNs} yields $K_3^N (s) \rightarrow 0$ $a.s$, as $N \rightarrow \infty$, for any $s > 0$. The Proposition is now a consequence of the two next Propositions.
\end{proof}
\begin{proposition}
 For any $s > 0$, $K_1^N(s) \longrightarrow 0$ in probability, as $N \rightarrow \infty$.
\end{proposition}
\begin{proof}
In this proof, we will use the following notations
\[\min_{\tau_{k}^N\le  r\le \tau_{k+1}^N} Y_{\delta_N}^{ref}(r)= Y_{\delta_N}^{ref}(r_Y^{k,N}) \quad \mbox{and} \quad \min_{\tau_{k}^N\le  r\le \tau_{k+1}^N} H_{\delta_N}(r)= H_{\delta_N}(r_H^{k,N}).\]
Let us define
\[U_{\delta_N}(s)= \int_0^s \int_{\delta_N}^\infty \left(z + \inf_{r\le u\le s} Y_{\delta_N}(u) - Y_{\delta_N}(r)\right)^+\Pi(dr,dz).\] 
We first note that $V_{U_{\delta_N}}[0,s]$, the total variation of $U_{\delta_N}$ on the interval $[0,s]$, satisfies
\begin{equation}\label{VT}
 \int_0^s \int_{\delta_N}^\infty z \Pi(dr,dz) \le V_{U_{\delta_N}}[0,s] \le  2 \int_0^s \int_{\delta_N}^\infty z \Pi(dr,dz). 
\end{equation}
However, we can rewrite \eqref{cHs} indexed by $\delta_N$ in the following form
$cH_{\delta_N}(s)= Y_{\delta_N}^{ref}(s)-U_{\delta_N}(s).$
It is not hard to obtain the following inequality
\begin{align}\label{ank}
U_{\delta_N}(r_Y^{k,N})= Y_{\delta_N}^{ref}(r_Y^{k,N})- cH_{\delta_N}(r_Y^{k,N}) &\le  \min_{\tau_{k}^N\le  r\le \tau_{k+1}^N} Y_{\delta_N}^{ref}(r) - \min_{\tau_{k}^N\le  r\le \tau_{k+1}^N} H_{\delta_N}(r) \nonumber\\&\le  Y_{\delta_N}^{ref}(r_H^{k,N}) - cH_{\delta_N}(r_H^{k,N})= U_{\delta_N}(r_H^{k,N})
\end{align}
Now, we have
\begin{align*}
K_1^N(s)&= \frac{1}{cN} \sum_{k=1}^{[\psi_{\delta_N}(N)s]} \left\{ (cH_{\delta_N}(\tau_k^N)- \min_{r\in [\tau_k^N,\tau_{k+1}^N] } cH_{\delta_N}(r) ) -  (Y_{\delta_N}^{ref}(\tau_k^N)- \min_{r\in [\tau_k^N,\tau_{k+1}^N] } Y_{\delta_N}^{ref}(r) ) \right\} \\
&=  \frac{1}{cN} \sum_{k=1}^{[\psi_{\delta_N}(N)s]} \left\{ -U_{\delta_N}(\tau_k^N) +\min_{r\in [\tau_k^N,\tau_{k+1}^N] } Y_{\delta_N}^{ref}(r) - \min_{r\in [\tau_k^N,\tau_{k+1}^N] } cH_{\delta_N}(r)  \right\} 
 =  \frac{1}{cN} \sum_{k=1}^{[\psi_{\delta_N}(N)s]}  \gamma_N(k),
 \end{align*}
 and \eqref{ank} implies that
$| \gamma_N(k)| \le \sup_{r_Y^{k,N} \le r\le r_H^{k,N}} | U_{\delta_N}(\tau_k^N) - U_{\delta_N}(r)|.$
Now from \eqref{VT}
\[\left| \sum_{k=1}^{[\psi_{\delta_N}(N)s]} \gamma_N(k) \right| \le  V_{U_{\delta_N}}[0,s] \le  2 \int_0^s \int_{\delta_N}^\infty z \Pi(dr,dz),\quad 
\mbox{which implies that} \quad 
 |K_1^N(s)| \le \frac{2}{cN} \int_0^s \int_{\delta_N}^\infty z \Pi(dr,dz). \]
The result follows easily from this estimate combined with assumption ${(\bf A)} $.
\end{proof}

For the proof of the next proposition, we need a basic result on Levy processes. Let us define
\[\Gamma(s)= \max_{0\le r\le s} (-Y_{\delta_N}(r)),\]
where $Y_{\delta_N}$ is a again a Lévy process with characteristic exponent $\psi_{\delta_N}$. The following result is Corollary 2, chapter VII in \cite{Ber}.
\begin{corollary}\label{COR1}
Since $\tau_1^N$ is an exponential random variable with parameter $\psi_{\delta_N}(N)$, independent of  $Y_{\delta_N}$, $\Gamma(\tau_1^N)$ has an exponential distribution with parameter $N$.
\end{corollary}
\begin{proposition}
 For any $s > 0$, $K_2^N(s) \longrightarrow s$ in probability, as $N \rightarrow \infty$.
\end{proposition}
\begin{proof}
We have
\[ K_2^N(s)=  \frac{\psi_{\delta_N}(N)}{cN^2}  \times \frac{1}{\psi_{\delta_N}(N)}\sum_{k=1}^{[\psi_{\delta_N}(N)s]}   N(Y_{\delta_N}^{ref}(\tau_k^N)- \min_{r\in [\tau_k^N,\tau_{k+1}^N] } Y_{\delta_N}^{ref}(r)).\]
We first notice that $0\le  e^{-\lambda} -1 + \lambda \le \lambda$, for all $\lambda\ge 0$, this implies
$\frac{\psi_{\delta_N}(N)}{cN^2} \longrightarrow 1, \ as \ N \rightarrow \infty.$
Let $\Gamma^{\prime}$ and $Y_{\delta_N}^{\prime}$ be independent copies of $\Gamma$ and $Y_{\delta_N}$ respectively. We notice that
\begin{align*}
Y_{\delta_N}^{ref}(\tau_k^N)- \min_{r\in [\tau_k^N,\tau_{k+1}^N] } Y_{\delta_N}^{ref}(r) &\stackrel{(d)}{=}  \left( \max_{r\in [\tau_k^N,\tau_{k+1}^N] } (-Y_{\delta_N}^{\prime}(r)) \right)  \wedge Y_{\delta_N}^{ref}(\tau_k^N) := \Gamma^{\prime}(\tau_{k+1}^N- \tau_{k}^N) \wedge Y_{\delta_N}^{ref}(\tau_k^N).
\end{align*}
Let $(\Xi_k)_{k\ge1}$ be an sequence of i.i.d random variables whose common law is that of $N\Gamma(\tau_1^N)$, such that in addition for any $k\ge1$, $\Xi_k$ and $\{Y_{\delta_N}^{ref}(r), r\le \tau_k^N\}$ are independent. We notice from Corollary \ref{COR1} that $\Xi_1$ has an standard exponential distribution. The Proposition is now a consequence the next lemma.
\end{proof}
\begin{lemma}
For any $s>0$,
\[ \frac{1}{\psi_{\delta_N}(N)}\sum_{k=1}^{[\psi_{\delta_N}(N)s]} \Xi_k \wedge NY_{\delta_N}^{ref}(\tau_k^N) \longrightarrow s \ in  \ probability,\ as \ N \rightarrow \infty.  \]
\end{lemma}
\begin{proof}
Let $\epsilon > 0$, which will eventually go to zero. Let $g_\epsilon , h_\epsilon : \R_+ \rightarrow \R$ be two functions defined by
\begin{equation*}
g_\epsilon (z)=
\left\{
    \begin{array}{ll}
  1, \quad \quad  \mbox{if} \quad z\le \epsilon, &\\\\ -\epsilon^{-1}z+2,   \quad \mbox{if} \quad  \epsilon<z \le 2\epsilon,&\\\\ 0,   \quad \mbox{if} \quad  z> 2\epsilon,

&
           \end{array}
           \right.
           \quad \mbox{and } \quad 
           h_\epsilon (z)=
\left\{
    \begin{array}{ll}
  0, \quad \quad  \mbox{if} \quad z\le \epsilon, &\\\\ -\epsilon^{-1}z-1,   \quad \mbox{if} \quad  \epsilon<z \le 2\epsilon,&\\\\ 1,   \quad \mbox{if} \quad  z> 2\epsilon.

&
           \end{array}
           \right.
\end{equation*}

It is not hard to see that
\[ I_1^N(s, \epsilon)+  I_2^N(s, \epsilon) \le  \frac{1}{\psi_{\delta_N}(N)}\sum_{k=1}^{[\psi_{\delta_N}(N)s]} \Xi_k \wedge NY_{\delta_N}^{ref}(\tau_k^N) \le J_1^N(s)+ J_2^N(s, \epsilon)+  J_3^N(s, \epsilon), \]
where
\begin{align*}
I_1^N(s, \epsilon)&=\frac{1-e^{-N\epsilon}}{\psi_{\delta_N}(N)}\sum_{k=1}^{[\psi_{\delta_N} (N)s]} h_\epsilon(Y_{\delta_N}^{ref}(\tau_k^N)) \quad [\mbox{where using the identity} \ \E(\Xi_k \wedge N\epsilon)=1-e^{-N\epsilon}],\\
I_2^N(s, \epsilon)&= \frac{1}{\psi_{\delta_N}(N)}\sum_{k=1}^{[\psi_{\delta_N} (N)s]} (\Xi_k \wedge N\epsilon - \E(\Xi_k \wedge N\epsilon)) \mathbf{1}_{\{Y_{\delta_N}^{ref}(\tau_k^N)> \epsilon \}}, \quad \quad J_1^N(s)=  \frac{1}{\psi_{\delta_N}(N)}\sum_{k=1}^{[\psi_{\delta_N} (N)s]} \Xi_k,\\
J_2^N(s, \epsilon)&= \frac{1}{\psi_{\delta_N}(N)}\sum_{k=1}^{[\psi_{\delta_N} (N)s]} g_\epsilon (Y_{\delta_N}^{ref}(\tau_k^N)),  \quad \mbox{and}\quad  J_3^N(s, \epsilon)= \frac{1}{\psi_{\delta_N}(N)}\sum_{k=1}^{[\psi_{\delta_N} (N)s]} (\Xi_k-1) \mathbf{1}_{\{Y_{\delta_N}^{ref}(\tau_k^N)\le \epsilon \}}.
\end{align*}
However, we first have
\begin{align*}
\E[I_2^N(s, \epsilon)]^2 \le  \frac{1}{(\psi_{\delta_N}(N))^2}\sum_{k=1}^{[\psi_{\delta_N} (N)s]} \mathbb{V}ar(\Xi_k \wedge N\epsilon) 
=&  \frac{1}{(\psi_{\delta_N}(N))^2}\sum_{k=1}^{[\psi_{\delta_N} (N)s]} (1- e^{-2N\epsilon}- 2N\epsilon e^{-N\epsilon})\longrightarrow 0,\ as \ N\rightarrow +\infty. 
\end{align*}
We can prove similarly that
$\E[J_3^N(s, \epsilon)]^2\longrightarrow 0,\ as \ N\rightarrow +\infty. $
Combining Lemma \ref{GG} and the fact that $Y^{ref}(r)>0$ $dr$ $a.s$, we deduce
\[  I_1^N(s, \epsilon) \xrightarrow[N\rightarrow \infty]{} \int_0^s h_\epsilon(Y_{\delta_N}^{ref}(r)) dr  \xrightarrow[ \epsilon \rightarrow 0]{} s, \quad \mbox{and} \quad
 J_2^N(s, \epsilon) \xrightarrow[N\rightarrow \infty]{} \int_0^s g_\epsilon(Y_{\delta_N}^{ref}(r) )dr  \xrightarrow[ \epsilon \rightarrow 0]{} 0.\]
In addition, we deduce from the law of large numbers that
$J_1^N(s) \longrightarrow s,\ as \ N\rightarrow +\infty.$
The desired result follows by combining the above arguments.
\end{proof}

Now, let us define
\begin{align*}
w_\delta^N(a,b)= \sup_{a\le r,s\le b, \ |s-r|\le \delta} |H^N(s)-H^N(r)|,
\quad \mbox{and} \quad w_{N,\delta}(a,b)= \sup_{a\le r,s\le b, \ |s-r|\le \delta} |H_{\delta_N}(s)-H_{\delta_N}(r)|.
\end{align*}
We shall also need below the
\begin{proposition}\label{HNtight}
For all $\epsilon > 0$,
$\lim_{\delta\rightarrow 0} \limsup_{N\rightarrow \infty} \P\left( w_\delta^N(a,b) \ge \epsilon \right)=0.$
\end{proposition}
\begin{proof}
We have
\[\left\{w_\delta^N(a,b) > \epsilon \right\} \subset \left\{\sup_{a<s<b} |\mathcal{H}_N(s)-s|+ \sup_{k_a\le k\le k_b}(\tau_{k+1}^N-\tau_{k}^N) >\delta \right\} \cup \{w_{N,3\delta}(a-\delta,b+\delta)>\epsilon\}, \]
where $\mathcal{H}_N(s)= \tau_{[\psi_{\delta_N}(N)K_N^{-1}(s)]}^N$, $k_a=[\frac{a}{\psi_{\delta_N}(N)}]$ and $k_b=[\frac{b}{\psi_{\delta_N}(N)}]-1$. 

So the result follows from both the two following facts : for each $\delta > 0$,
\begin{equation}\label{convunif}
\forall \delta>0, \ \P\left(\sup_{a<r<b} |\mathcal{H}_N(r)-r| >\delta \right) \rightarrow 0, \ as \ N\rightarrow \infty,
\end{equation}
\begin{equation}\label{tightness}
\forall \epsilon>0, \  \lim_{\delta\rightarrow 0} \limsup_{N\rightarrow \infty} \P\left( w_{N, 3\delta}(a-\delta,b-\delta) > \epsilon \right)=0.
\end{equation}
PROOF OF \eqref{convunif}.  It follows from a combination of Propositions \ref{tauNs} and \ref{KNs} that $\mathcal{H}_N(s) \rightarrow s$ in probability, for any $s>0$. Moreover for any $N, s \rightarrow \mathcal{H}_N(s)$ is increasing. Let $M\ge1$ and $a= s_0<s_1<\cdots <s_M=b$ be such that $\sup_{0\le i\le M-1} (s_{i+1}- s_i) \le \delta/2$. For any $\epsilon>0$, we can choose $N_\epsilon$ large enough such that for all $N\ge N_{\epsilon}$, 
$\P\left( \bigcap_{i=0}^{M}\left\{ |\mathcal{H}_N(s_i)-s_i | \le \delta/2\right\}  \right) \ge 1-\epsilon.$
But for any $s_I \le s \le s_{i+1}$, on the event $\bigcap_{i=0}^{M}\left\{ |\mathcal{H}_N(s_i)-s_i | \le \delta/2\right\}$,
\[s-\delta\le s_i- \delta/2\le  \mathcal{H}_N(s_i) \le \mathcal{H}_N(s) \le \mathcal{H}_N(s_{i +1}) \le s_{i +1}+ \delta/2\le s+ \delta, \]
hence we have shown that for $N \ge N_\epsilon$, the following property equivalent to  \eqref{convunif}
\[ \P\left(\sup_{a<s<b} |\mathcal{H}_N(s)-s| \le \delta \right)\ge 1-\epsilon. \]

PROOF OF \eqref{tightness}. Since $H_{\delta_N}(s) \rightarrow H(s)$ uniformly in $[(a-1)\vee 0, b+1]$ in probability, hence it converges in law in $\mathcal{C}([(a-1)\vee 0, b+1])$, hence the sequence $\{ H_{\delta_N}, N\ge1\}$ is tight in $\mathcal{C}([(a-1)\vee 0, b+1])$, from which \eqref{tightness} follows.
\end{proof}

$\bf{Proof \ of \  Theorem \ \ref{THP} :}$ From \eqref{ega}, we have 
\begin{align*}
|H^N(s)-H(s)|\le |H^N(s)-H^N(K_N(s))|+ \left|H_{\delta_N} \left(\tau_{[\psi_{\delta_N}(N)s]}^N \right)-H \left(\tau_{[\psi_{\delta_N}(N)s]}^N \right) \right| + \left|H \left(\tau_{[\psi_{\delta_N}(N)s]}^N \right) -H(s) \right|.
\end{align*}
A combination of Propositions \ref{KNs} and \ref{HNtight} implies that the first term on the right tends to $0$ in probability, as $N\rightarrow +\infty$. Since $H_{\delta_N} \rightarrow H$ $a.s.$ locally uniformly in $s$, and from Proposition \ref{tauNs}, $\tau_{[\psi_{\delta_N}(N)s]} \rightarrow s$ $a.s.$, the second term tends to $0$ a.s. Finally the last term tends to $0$ $a.s.$ thanks again to Proposition \ref{tauNs} and the continuity of $H$.

We have juste proved that for each $s>0$, $H^N(s) \longrightarrow H(s)$ in probability, as $N \rightarrow +\infty$. Since from Proposition \ref{HNtight}, $H$ is tight in $\mathcal{C}([0, s])$ for all $s > 0$, the convergence is locally uniform in $s$.
$\hfill \blacksquare$

$\mathbf{Ackownledgement}$.
I would like to give my sincere thanks to my Phd supervisor Professor Etienne Pardoux for his englightening discussions and helpful suggestions. I would also like to thank Professor Thomas Duquesne for his suggestion to study the approximation studied in the present paper. 
\frenchspacing
\bibliographystyle{plain}

\end{document}